\documentclass[12pt]{article}
\usepackage[utf8]{inputenc}
\usepackage{lmodern}
\usepackage{graphicx, xcolor, mathrsfs}
\usepackage{amsfonts}
\usepackage{stmaryrd}
\usepackage{blkarray, bigstrut}
\usepackage{mathtools}
\usepackage[nice]{nicefrac}
\usepackage{amsmath,amsthm,amssymb,bbm}
\usepackage[mathlines]{lineno}
\usepackage{mathabx}
\usepackage{accents}

\usepackage{enumerate}
\usepackage{enumitem}
\usepackage[colorlinks=true,citecolor=black,linkcolor=black,urlcolor=blue]{hyperref}
\usepackage[top=.9in, bottom=.9in, left=.5in , right=.5in]{geometry}
\usepackage{caption}
\usepackage{xifthen} 
\usepackage{wrapfig}
\usepackage[numbers, square]{natbib}
\usepackage{changes}
\usepackage{comment}
\usepackage{float}
\usepackage{ifthen}
\mathtoolsset{showonlyrefs}
\usepackage{tikz}
\usepackage{caption}                       
\usetikzlibrary{arrows}
\usetikzlibrary{graphs,graphs.standard}
\usetikzlibrary{positioning,arrows.meta}
\usetikzlibrary{shapes.multipart}
\usetikzlibrary{arrows}
\usetikzlibrary{automata}
\definecolor{arrowblue}{RGB}{0,0,0}  

\usepackage[splitrule]{footmisc} 
\usepackage{caption}
\newtheorem{thm}{Theorem}[section]
\newtheorem*{thm*}{Theorem}

\theoremstyle{definition}

\DeclareMathOperator{\exponentialrv}{Exp}

\newcommand{\Exp}[1]{\exponentialrv\left( #1 \right)}

\DeclareMathOperator{\degg}{deg}

\newcommand{\outdeg}[1]{\ensuremath{\degg^{+}(#1)}}

\title{A counter-example to persistence in generalised preferential attachment trees}
\author{Tejas Iyer\footnote{Weierstrass Institute for Applied Analysis and Stochastics, Anton-Wilhelm-Amo-Str. 39, 10117 Berlin, Germany. \\ E-Mail: \href{mailto:tejas.iyer@wias-berlin.de}{tejas.iyer@wias-berlin.de} }}
\date{\today}

\begin{document}

\maketitle
\abstract{Consider a generalised preferential attachment tree with attachment function $f$, that is a random tree, where at each time-step a node connects to an existing node $v$ with probability proportional to $f(\deg(v))$, where $\deg(v)$ denotes the degree of the node in the existing tree. We provide a counter-example to a conjecture of the author asserting that under the assumption $\sum_{j=1}^{\infty} \frac{1}{f(j)^2} < \infty$ there is a persistent hub in the model, that is, a single node that has the maximal degree for all but finitely many time-steps. The counter-example is a minor modification of a related counter-example due to Galganov and Ilienko. 
}

\section{Introduction \& main result}

Generalised preferential attachment trees are a model of randomly growing (directed) trees constructed as follows. One is given a fixed function $f\colon \mathbb{N}_0 :=\mathbb{N} \cup \{0\} \rightarrow (0, \infty)$ and starts with a tree $\mathcal{T}_0$ consisting of a single node $0$. Then, for $n \in \mathbb{N}_0$, to construct the tree $\mathcal{T}_{n+1}$ at the time-step $n+1$, a new node labelled $n+1$ is added, and connects to an existing node $j \in \mathcal{T}_n$ with probability 
\[
\frac{f(\deg^{+}(j, \mathcal{T}_{n}))}{\sum_{j=0}^{n} f(\deg^{+}(j, \mathcal{T}_{n}))},  
\]
with the edge directed outwards from $j$, i.e., $j \rightarrow n+1$. We say a node $u \in \mathbb{N}_0$ is a \emph{persistent hub} if $\outdeg{u, \mathcal{T}_{n}} = \max_{v \in \mathcal{T}_{n}} \outdeg{v, \mathcal{T}_{n}}$ for all but finitely many $n \in \mathbb{N}_0$. Based on previous evidence from works such as~\cite{dereich-morters-persistence, banerjee-bhamidi, galashin}, it is conjectured~\cite[Remark~2.11.]{iyer2024persistenthubscmjbranching} that under the assumption $\sum_{j=0}^{\infty} \frac{1}{f(j)^2}$ with probability one there is a persistent hub in the sequence of trees $(\mathcal{T}_{n})_{n \in \mathbb{N}_0}$. The main result is a counter-example to this claim:

\begin{thm} \label{thm:counter-example}
    There exists a function $f\colon \mathbb{N}_0 \rightarrow (0,\infty)$ satisfying $\sum_{j=0}^{\infty} \frac{1}{f(j)^2} < \infty$, such that, almost surely, for the associated generalised preferential attachment trees $(\mathcal{T}_{n})_{n \in \mathbb{N}_0}$ there is no persistent hub. 
\end{thm}

Theorem~\ref{thm:counter-example} is somewhat surprising, in part because the condition $\sum_{j=0}^{\infty} \frac{1}{f(j)^2} < \infty$ is known to be necessary and sufficient for the similar notion of leadership in non-linear P\'olya urn schemes~\cite{oliveira-brownian-motion, leadership}. 
We emphasise that this counter-example is a modification of the similar counter-example due to Galganov and Ilienko~\cite{galganov-ilienko}. This construction also provides a counter-example to the more general conjecture of the author concerning persistent hubs in CMJ branching processes with independent increments~\cite[Remark~2.6]{iyer2024persistenthubscmjbranching}. 

\section{Proof}
In order to prove the result, we use the standard embedding of the process as a Crump-Mode-Jagers (CMJ) process $(\mathscr{T}_{t})_{t \geq 0}$. \emph{Individuals} in the process are labelled by elements of the infinite \emph{Ulam-Harris} tree $\mathcal{U} : = \bigcup_{n \geq 0} \mathbb{N}^{n}$. The set $\mathbb{N}^{0} := \{\varnothing\}$ represents the ancestral \emph{root} individual $\varnothing$. We denote other elements $u \in \mathcal{U}$ as a tuple $u_1 \cdots u_m$, which may be interpreted as the $u_m$th child of $u_1 \cdots u_{m-1}$. We label elements of $\mathcal{U}$ with values in $[0, \infty]$, representing \emph{birth-times}. Associated with each $u \in \mathcal{U}$ is a collection of exponentially distributed random variables $(X(uj))_{j \in \mathbb{N}} \in [0,\infty)^\mathbb{N}$, where each $X(uj) \sim \Exp{f(j-1)}$. We think of $X(uj)$ as the \emph{displacement} or \emph{waiting time} between the $(j-1)$th and $j$th child of $u$. We then define 
the random function $\mathcal{B}\colon \mathcal{U} \rightarrow [0, \infty]$ recursively as follows:  
\[\mathcal{B}(\varnothing) : = 0 \quad \text{and for $u \in \mathcal{U}, i \in \mathbb{N}$,} \quad \mathcal{B}(ui) := \mathcal{B}(u) + \sum_{j=1}^{i} X(uj).\]
For each $u \in \mathcal{U}$ we think of the value $\mathcal{B}(u)$ as its `birth time'. For each $t \in [0, \infty]$, we set $\mathscr{T}_{t} = \{x \in \mathcal{U} \colon \mathcal{B}(x) \leq t\}$ and note that $\mathscr{T}_{t}$ may be interpreted as a directed tree in the natural way - connecting elements $u \in \mathcal{U}$ with an edge directed outwards to its children $(u \ell)_{\ell \in \mathbb{N}}$ in $\mathcal{U}$. In relation to the process $(\mathscr{T}_{t})_{t \geq 0}$, we define the stopping times $(\tau_{k})_{k \in \mathbb{N}_{0}}$ such that 
 \begin{equation} \label{eq:tau-k-def} 
 \tau_{k} := \inf\{t \geq 0\colon |\mathscr{T}_{t}| \geq k\}.
 \end{equation}
Note that the values $(\tau_{k})_{k \in \mathbb{N}_0}$ describe the times in which changes, or `jumps' occur in the process $(\mathscr{T}_{t})_{t \geq 0}$. It is well-known, by applying properties of exponential random variables, that $(\mathscr{T}_{\tau_{n}})_{n \in \mathbb{N}_{0}}$ and the process of generalised preferential attachment trees $(\mathcal{T}_{n})_{n \in \mathbb{N}_0}$ are equal in distribution (up to, say, re-labelling of vertices). 

For $t \in [0, \infty]$, denote by $\xi([0,t])$ the random variable indicating the number of offspring of the root generated in the process $(\mathscr{T}_{t})_{t \geq 0}$ up to time $t$. 

\begin{proof}[Proof of Theorem~\ref{thm:counter-example}]
    Note that if $(f(j))_{j \in \mathbb{N}_0}$ is a sequence such that $\sum_{j=0}^{\infty} \frac{1}{f(j)} = \infty$, and yet \[\lim_{n \to \infty} \tau_{n} =: \tau_{\infty} < \infty\] almost surely, then it is impossible for $(\mathcal{T}_{n})_{n \in \mathbb{N}_0}$ to contain a persistent hub. Indeed, $\sum_{j=0}^{\infty} \frac{1}{f(j)} = \infty$ ensures, by standard criteria for pure-birth processes, that $\xi([0,t]) < \infty$ for all $t \in [0, \infty)$ which means, by a union bound over the countable set $\mathcal{U}$, that the degree of every node in $\mathcal{T}_{\infty}:= \bigcup_{n \in \mathbb{N}_0} \mathcal{T}_n$ is finite almost surely. If the maximal degree in $\mathcal{T}_{\infty}$ was bounded from above, by $L$ say, this would imply the existence of infinitely many individuals whose ancestors all have at most $L$ children, which contradicts~\cite[Proposition~4.4]{inhom-sup-pref-attach} (such individuals are termed $L$-moderate there). Therefore, the maximal degree of $\mathcal{T}_{n}$, and hence the degree of any persistent hub must diverge to infinity.  
    
    We now make a minor modification to the construction of Galganov and Ilienko, to construct a function $f$ such that
    $\sum_{j=0}^{\infty} \frac{1}{f(j)} = \infty$, $\sum_{j=0}^{\infty} \frac{1}{f(j)^2} < \infty$ and $\tau_{\infty} < \infty$ almost surely. To do so, we define the values $f(0), f(1), \ldots$, by 
    \[
    1, \underbrace{\alpha_{1}, \cdots, \alpha_1}_{d_1 \text{ times}}, 2, \underbrace{\alpha_{2}, \cdots, \alpha_2}_{d_2 \text{ times}}, 3, \underbrace{\alpha_{2}, \cdots, \alpha_3, }_{d_3 \text{ times}} \cdots, 
    \]
    where $d_{k} := 4^{k^2}$ and $\alpha_{k} := 2^{2^{k^3}}$. Note that $\sum_{j=0}^{\infty} \frac{1}{f(j)} > \sum_{j=1}^{\infty} \frac{1}{j} = \infty$, whilst $\sum_{j=0}^{\infty} \frac{1}{f(j)^{2}} = \sum_{k=1}^{\infty} \frac{1}{k^2} + \sum_{k=1}^{\infty} \frac{4^{k^2}}{2^{2^{k^3}}} < \infty$. 
    It remains to show that $\tau_{\infty} < \infty$ almost surely. For this, note that Galganov and Ilienko have already shown that this occurs for the process with smaller rates 
    \[
    1, \underbrace{\alpha_{1}, \cdots, \alpha_1}_{d_1 \text{ times}}, 1, \underbrace{\alpha_{2}, \cdots, \alpha_2}_{d_2 \text{ times}}, 1, \underbrace{\alpha_{3}, \cdots, \alpha_3, }_{d_3 \text{ times}} \cdots, 
    \]
    whence, by a standard coupling argument (coupling the pure-birth processes associated with individuals in the processes), we also get $\tau_{\infty} < \infty$ in this process. 
\end{proof}

\section*{Acknowledgements}
TI is funded by Deutsche Forschungsgemeinschaft (DFG) through DFG Project no. $443759178$.

\bibliographystyle{abbrv}
\bibliography{refs}

\end{document}